\documentclass[11pt]{amsart}
\usepackage[utf8]{inputenc}
\usepackage{amsmath,amsthm,amsfonts,amssymb,amscd,mathtools}
\usepackage{verbatim}
\usepackage{mathrsfs}
\usepackage{multicol}
\usepackage{tabularx}
\usepackage[usenames,dvipsnames]{color}
\usepackage{enumerate}
\usepackage{graphicx}
\usepackage{color,xcolor}
\usepackage{bbm}
\usepackage[all]{xy}
\usepackage{hyperref}
\usepackage[capitalise]{cleveref}
\usepackage{pb-diagram,tikz-cd}

\theoremstyle{definition}
\newtheorem{thm}{Theorem}[subsection]
\newtheorem{prop}[thm]{Proposition}
\newtheorem{cor}[thm]{Corollary}

\newtheorem{lem}[thm]{Lemma}

\numberwithin{equation}{subsection}

\def\bs#1{\boldsymbol{#1}}
\def\bbs#1{\boldsymbol{\bar #1}}
\def\lie#1{\mathfrak{#1}}
\def\tlie#1{\tilde{\mathfrak{#1}}}

\def\endd{\hfill$\diamond$}

\DeclareMathOperator\wt{{\rm wt}}

\begin{document}

	\title[Totally ordered graphs and prime factorization]{Totally ordered pseudo $q$-factorization graphs\\ and prime factorization}
	
	\author{Matheus Brito}
	\address{Departamento de Matemática, Universidade Federal do Paraná, Curitiba - PR - Brazil, 81530-015}
	\email{mbrito@ufpr.br, cristiano.clayton@ufpr.br}
	\thanks{}
	\author{Adriano Moura}
	\address{Departamento de Matemática, Universidade Estadual de Campinas, Campinas - SP - Brazil, 13083-859.}	\email{aamoura@unicamp.br}
	\thanks{The work of M.B. and A.M. was partially supported CNPq grant 405793/2023-5. The work of A.M. was also partially supported by Fapesp grant 2018/23690-6. }
	\author{Clayton Silva}
	\thanks{C. Silva is grateful to CNPq for financial support (grant 100161/2024-3 INCTMat-IMPA) during his postdoctoral research internship at UFPR}
	
	\begin{abstract}
		In an earlier publication, the last two authors showed that a finite-dimensional module for a quantum affine algebra of type $A$ whose $q$-factorization graph is totally ordered is prime. In this paper, we continue the investigation of the role of totally ordered pseudo $q$-factorization graphs in the study of the monoidal structure of the underlying abelian category. We introduce the notions of modules with (prime) snake support and of maximal totally ordered subgraphs decompositions. Our main result shows that modules with snake support have unique such decomposition and that it determines the corresponding prime factorization. Along the way, we also prove that prime snake modules (for type $A$) can be characterized as the modules for which every pseudo $q$-factorization graph is totally ordered.
	\end{abstract}
	
	\maketitle

	\section{Introduction}
	
	A major problem in the realm of finite-dimensional representations of quantum affine algebras is that of describing the factorizations of a given simple module as a tensor product of prime ones. Certainly, the most successful theoretical approach so far is the connection with cluster algebras. Still, given a specific example, it is not such a simple task to find the appropriate answer using this approach. Also, it might be that, for certain classes of modules, other approaches may lead to more efficient algorithms for finding such factorizations. This is the spirit of our recent publications, which is further developed here.
	
	The concept of pseudo $q$-factorization graphs was recently introduced in \cite{ms:to} as a combinatorial language which is suited for capturing certain properties of Drinfeld polynomials. Among all pseudo $q$-factorization graphs afforded by a given Drinfeld polynomial, there is a maximal (the fundamental factorization graph) and a minimal (the $q$-factorization graph). Using certain known representation theoretic facts about tensor products of Kirillov Reshetikhin modules and qcharacters, combined with special topological/combinatorial properties of the underlying $q$-factorization graphs, it was shown in \cite{ms:to} that, for algebras of type $A$, modules associated to totally ordered $q$-factorization graphs are prime. This provided a first glimpse into the role of totally ordered pseudo $q$-factorization graphs in the study of prime factorizations. We have also used certain totally ordered pseudo $q$-factorization graphs to construct strongly real modules in \cite{bms}.
	
	A connected pseudo $q$-factorization graph with at most two vertices is totally  ordered and the underlying module is certainly prime. The prime factorization of a module whose $q$-factorization graph has 3 vertices was described in \cite{ms:3tree} and we recall it in  \Cref{t:3lineprime} below. In particular, there is a complete characterization of prime modules whose $q$-factorization graphs have 3 vertices, but are not totally ordered. The guiding line of the present paper is the following question: Can we describe classes of modules for which the prime factorizations can be found from the (purely combinatorial) study of chains of maximal totally ordered subgraphs (mtos for short) of a given pseudo $q$-factorization graph? Our main result, \Cref{t:orddecss}, describes such a family: that of modules with snake support. We also show that \Cref{t:orddecss} includes, as a particular case, the main result of \cite{bc:hokr} on the prime factorization of modules supported on a single node of the underlying finite type Dynking diagram. In other words, the study of mtos describes the prime factorizations of modules with snake support and, in particular, of any module afforded by monochromatic pseudo $q$-factorization graphs.

	The concepts of snakes and prime snakes were introduced in \cite{my:pathB} for algebras of types $A$ and $B$. We review the formal definition for type $A$ in \Cref{ss:snakemods}. It is immediate that the original definition of prime snakes is equivalent to requesting that the associated fundamental factorization graph is totally ordered. We provide another characterization in \Cref{t:ps<=>toG}: a simple module is a prime snake module if and only if all of its pseudo $q$-factorization graphs are totally ordered. This characterization could serve as a type independent conceptual definition of prime snakes.  We give a few first steps on the discussion about switching between choices of pseudo $q$-factorization graphs in \Cref{ss:fuse}. We keep finding examples where intermediate pseudo $q$-factorization graphs are more suited for answering different questions than the two extremal graphs. It is not clear to us how to answer the question: what is the most convenient graph to work with? We leave further studies in this direction for the future. The initial discussion we make here provides a proof for  \Cref{t:ps<=>toG}. 
	
	Let us explain the statement of \Cref{t:orddecss}. Given a Drinfeld polynomial $\bs\pi$,  let $\bbs\pi$ be the Drinfeld polynomial having exactly one copy of each fundamental factor of $\bs\pi$. We say $\bs\pi$ has snake support if $\bbs\pi$ arises from a prime snake. If that is the case, \Cref{t:orddecss} says the fundamental factorization graph $G_f(\bs\pi)$ of $\bs\pi$ admits a unique (up to isomorphism) mtos-quochain and, moreover, if $G_1,\dots,G_l$ is the multicut of $G_f(\bs\pi)$ associated with such a quochain, then
	\begin{equation*}
		V(\bs\pi) \cong V(\bs\pi_1)\otimes\cdots\otimes V(\bs\pi_l)
	\end{equation*}
	is the unique prime factorization of the simple module $V(\bs\pi)$, where $\bs\pi_j$ is the Drinfeld polynomial associated to $G_j$. The notion of quochains was introduced in \cite{bms} and reviewed in \Cref{ss:graphmorp} below, where we also introduce the notion of isomorphic quochains. In summary, the aforementioned multicut is obtained as follows: $G_1$ is an mtos of $G_f(\bs\pi)$, $G_2$ is an mtos in the graph obtained from $G_f(\bs\pi)$ by deleting the vertices and arrows related to $G_1$, and so on. For any such choice of multicut, the corresponding factors $V(\bs\pi_j)$ are isomorphic.
	
	The basic background and notation for the main statements are collected in Sections \ref{ss:basenot} to \ref{ss:graphmorp}. The main results are stated in Sections \ref{ss:snakemods} and \ref{ss:sspf}. As mentioned above, \Cref{t:orddecss} can be seen as a generalization of the main result of \cite{bc:hokr}. Since the terminology used in \cite{bc:hokr} is different than the one used here, we also provide a more explicit comparison in  \Cref{ss:sspf}. In \Cref{ss:evmcg}, we make some comments about modules for which the prime factorizations do not arise from mtos-quochains. Further background needed for the proofs and the main proofs are given in \Cref{s:pfs}.

	\section{Basic Notation and The Main Statements}\label{s:baseback}
	
	Throughout the paper, let  $\mathbb Z$ denote the set integers. Let also $\mathbb Z_{\ge m} ,\mathbb Z_{< m}$, etc.,  denote the obvious subsets of $\mathbb Z$. Given a ring $\mathbb A$, the underlying multiplicative group of units is denoted by $\mathbb A^\times$. 
	The symbol $\cong$ means ``isomorphic to''. We shall use the symbol $\diamond$ to mark the end of remarks, examples, and statements of results whose proofs are postponed. The symbol \qedsymbol\ will mark the end of proofs as well as of statements whose proofs are omitted. 
	
	\subsection{Quantum Algebras and Their Finite-Dimensional Modules}\label{ss:basenot}
	Although we use the basic notation as in \cite{ms:to,ms:3tree,bms}, for the reader's convenience, we review it here.  
	
	Let $\lie g$ be a simple Lie algebra of type $A_n$ over $\mathbb C$ and let $I$ be the set of nodes of its Dynkin diagram. We let $x_i^\pm, h_i, i\in I$, denote generators as in Serre's Theorem and let $\lie g=\lie n^-\oplus\lie h\oplus \lie n^+$ be the corresponding triangular decomposition. The symbols $R, R^+, Q,Q^+,P,P^+$ will stand for, respectively, the sets of roots, positive roots, root lattice, the monoid generated by the positive roots, the weight lattice, and the set of integral dominant weights. The fundamental weights and simple roots will be denoted by $\omega_i, \alpha_i,i\in I$.   For $i\in I$, let $i^*= w_0(i)$, where $w_0$ is the Dynkin diagram automorphism induced by the longest element of the Weyl group.   For $i,j\in I$, we let $[i,j]$ denote the connected subgraph having $i,j$ as boundary nodes, while the set of boundary nodes of $J\subseteq I$ is denoted by $\partial J$. We also let $d(i,j)=\#[i,j]-1$ and $d(J,K)=\min\{d(j,k): j\in J, k\in K\}$ for $J,K\subseteq I$.
	
	We let $U_q(\tlie g)$  be the quantum affine (in fact loop) algebra over an algebraically closed field of characteristic zero $\mathbb F$, where $q\in\mathbb F^\times$ is not a root of unity.  The generators are denoted by $x_{i,r}^\pm, k_i^{\pm 1}, h_{i,s}, i\in I, r,s\in\mathbb Z, s\ne 0$. The subalgebra generated by $x_i:=x_{i,0}^\pm, k_i^{\pm 1}, i\in I$ is a Hopf subalgebra of $U_q(\tlie g)$ isomorphic to the Drinfeld-Jimbo quantum group $U_q(\lie g)$.

	For  $i\in I$, $a\in \mathbb Z$, we let $\bs\omega_{i,a}$ denote the corresponding fundamental $\ell$-weight, which is the Drinfeld polynomial whose unique non-constant entry is equal to $1-q^au\in\mathbb F[u]$. 
	We let $\mathcal P^+$ denote the multiplicative monoid generated by such elements, with identity element denoted by $\bs 1$, while $\mathcal P$ denotes the corresponding abelian group. We shall say $\bs\omega_{i,a}$ occurs in $\bs\pi\in\mathcal P^+$ if it appears in a reduced expression for $\bs\pi$ as a product of fundamental $\ell$-weights and set
	\begin{equation*}
		\sup(\bs\pi) = \{i\in I: \bs\omega_{i,a} \text{ occurs in } \bs\pi \text{ for some } a\in\mathbb Z\}.
	\end{equation*}
	Let $\mathcal P_i^+ = \{\bs\pi\in\mathcal P^+:\sup(\bs\pi)=\{i\}\}$.

	Let $\mathcal C$ be the full subcategory of that of fintie-dimensional $U_q(\tlie g)$-modules whose simple factors have highest $\ell$-weights in $\mathcal P^+$ and, hence, $\ell$-weights in $\mathcal P$. Thus, a finite-dimensional $U_q(\tlie g)$-module $V$ is in $\mathcal C$ iff
	$$V=\bigoplus_{\bs\varpi\in\mathcal P}^{} V_{\bs\varpi}$$
	where $V_{\bs\varpi}$ is the $\ell$-weight space of $V$ associated to $\bs\varpi\in\mathcal P$. Set
	\begin{equation*}
		\wt_\ell(V) = \{\bs\varpi\in\mathcal P: V_{\bs\varpi}\ne 0\}. 
	\end{equation*}

	For $\bs\pi\in\mathcal P^+$, $V(\bs\pi)$ will denote a simple $U_q(\tlie g)$-module whose highest $\ell$-weight is $\bs\pi$. Since $\mathcal C$ is a monoidal category, the notion of prime objects is defined. Moreover, since $V(\bs 1)$, the one-dimensional trivial representation, is the unique invertible object in $\mathcal C$, it follows that $V\in\mathcal C$ is prime iff 
	\begin{equation*}
		V\cong V_1\otimes V_2 \quad\Rightarrow\quad V_j\cong V(\bs 1) \quad\text{for some}\quad j\in\{1,2\}.
	\end{equation*}
	Every simple object in $\mathcal C$ admits a decomposition as a tensor product of simple prime modules and a simple object $V$ is said to be real if $V\otimes V$ is simple.

	For an object $V\in\mathcal C$, let $V^*$ and ${}^*V$ be the dual modules to $V$ such that the usual evaluation maps
	\begin{equation*}
		V^*\otimes V \to \mathbb F  \quad\text{and}\quad V\otimes {}^*V\to\mathbb F
	\end{equation*}
	are module homomorphisms (cf. \cite[Section 2.6]{ms:to}). Then $({}^*V)^*\cong V\cong {}^*(V^*)$ and  $(V_1\otimes V_2)^*\cong V_2^*\otimes V_1^*$. 
	The Hopf algebra structure on $U_q(\tlie g)$ is chosen so that, if $V=V(\bs\pi)$, then $V^*\cong V(\bs\pi^*)$, where $\bs\pi\mapsto \bs\pi^*$ is the group automorphism of $\mathcal P$ determined by
	\begin{equation*}
		\bs\omega_{i,a}^* = \bs\omega_{i^*,a-\check h}. 
	\end{equation*}
	Similarly, considering the automorphism determined by ${}^*\bs\omega_{i,a} = \bs\omega_{i^*,a+\check h}$, it follows that ${}^*V(\bs\pi)\cong V({}^*\bs\pi)$. 
	
	Given $i\in I, a\in\mathbb Z, r\in\mathbb Z_{\ge 0}$,  define 
	\begin{equation*}
		\bs\omega_{i,a,r} = \prod_{p=0}^{r-1} \bs\omega_{i,a+{r-1-2p}}.
	\end{equation*}
	These are the Drinfeld polynomials of the Kirillov-Reshetikhin modules and, hence, we refer to them as Drinfeld polynomials of KR.  The set of all such polynomials will be denoted by $\mathcal{KR}$. 
	Every $\bs\pi\in\mathcal P^+$ can be written uniquely as a product of KR type polynomials such that, for every two factors supported at $i$, say $\bs\omega_{i,a,r}$ and $\bs\omega_{i,b,s}$, the following holds
	\begin{equation}\label{e:defqfact}
		a-b \notin\mathscr R_i^{r,s} :=\{r+s-2p: 0\le p<\min\{r,s\}\}. 
	\end{equation}
	Such factorization is said to be the  $q$-factorization of $\bs\pi$ and the corresponding factors are called the $q$-factors of $\bs\pi$.
	By abuse of language, whenever we mention the set of $q$-factors of $\bs\pi$ we actually mean the associated multiset of $q$-factors counted with multiplicities in the $q$-factorization.  It is often convenient to work with factorizations in KR type polynomials which not necessarily satisfy \eqref{e:defqfact}. Such a factorization will be referred to as a pseudo $q$-factorization and the associated factors as the corresponding $q$-factors of the factorization.
	
	Given $(i,r),(j,s)\in I\times  \mathbb Z_{>0}$ and $a,b\in\mathbb Z$, set
	\begin{equation}\label{e:defredset'}
		\mathscr R_{i,j}^{r,s} = \{r+s+d(i,j)-2p: - d([i,j],\partial I)\le p<\min\{r,s\}  \}. 
	\end{equation}
	Note
	\begin{equation}\label{e:sl2inRij}
		\mathscr R_{i}^{r,s}\subseteq \mathscr R_{i,i}^{r,s}.
	\end{equation} 
	It is well-known that
	\begin{equation}\label{e:defredset}
		V(\bs\omega_{i,a,r})\otimes V(\bs\omega_{j,b,s}) \text{ is reducible}\qquad\Leftrightarrow\qquad |a-b| \in \mathscr R_{i,j}^{r,s}.
	\end{equation}
	Moreover, in that case,
	\begin{equation}\label{e:krhwtp}
		V(\bs\omega_{i,a,r})\otimes V(\bs\omega_{j,b,s}) \text{ is  highest-$\ell$-weight}\qquad\Leftrightarrow\qquad a>b.
	\end{equation}
	If $r=s=1$, we simplify notation and write $\mathscr R_{i,j}$ for $\mathscr R_{i,j}^{1,1}$.

	\subsection{Pseudo q-Factorization Graphs}\label{ss:qfgraphs}
	The notion of a pseudo $q$-factorization graph was introduced in \cite{ms:to}. We now review the rephrased definition given in \cite{bms}.
	
	Let $G=(\mathcal V,\mathcal A)$ be a digraph with vertex set $\mathcal V$ and arrow set $\mathcal A\subseteq \mathcal V\times\mathcal V$. If $a=(v,w)\in\mathcal A$, we set $h(a)=w$ and $t(a)=v$ (the head and the tail of $a$). A pseudo q-factorization map over $G$ is a map $\mathcal F:\mathcal V\to\mathcal{KR}$ such that
	\begin{equation}
		\mathcal F(v) = \bs\omega_{i,a,r} \ \ \text{and}\ \ \mathcal F(w) = \bs\omega_{j,b,s} \quad\Rightarrow\quad \Big[ (v,w)\in\mathcal A \quad\Leftrightarrow\quad a-b \in\mathscr R_{i,j}^{r,s} \Big].
	\end{equation}
	In particular, if such a map exists, $G$ does not contain loops nor oriented cycles. A pseudo q-factorization graph is a digraph equipped with a pseudo q-factorization map. We shall say that a pseudo $q$-factorization map $\mathcal F$ over $G$ is fundamental if, for all $v\in\mathcal V$, $\mathcal F(v)$ is a fundamental $\ell$-weight.  The corresponding pseudo $q$-factorization graph will then be referred to as a fundamental factorization graph. Recall \eqref{e:defqfact} and \eqref{e:sl2inRij}. We shall say  $\mathcal F$ is a $q$-factorization map if $\mathcal F(v)$ is a $q$-factor of 
	\begin{equation}\label{e:polyofgraph}
		\bs\pi_{\mathcal F} := \prod_{v\in\mathcal V} \mathcal F(v)\in\mathcal P^+.
	\end{equation}
	In that case, the corresponding pseudo $q$-factorization graph will be referred to as a $q$-factorization graph.
	By abuse of notation, we shall identify $v\in\mathcal V$ with $\mathcal F(v)$, so we can shorten the above to $\bs\pi_{\mathcal F} =  \prod_{v\in\mathcal V} v$. Moreover, we shall abuse of language and simply say ``$G$ is a pseudo $q$-factorization graph'' with no mention to the structure data $(\mathcal V,\mathcal A,\mathcal F)$ and then write $\bs\pi_G$ instead of $\bs\pi_{\mathcal F}$. Despite the lack of accuracy, this should not cause contextual confusion and should be most often beneficial for the conciseness of the text. We shall also say $G$ is a pseudo $q$-factorization graph over $\bs\pi$ if $\bs\pi_G=\bs\pi$. 
	If $H\triangleleft G$. i.e., if $H=(\mathcal U, \mathcal A')$ is a subgraph of $G$, we set
	\begin{equation}\label{e:polyofsubgraph}
		\bs\pi_H = \bs\pi_{\mathcal F|_{\mathcal U}}.
	\end{equation}

	Conversely, given any map $\mathcal F:\mathcal V\to\mathcal{KR}$ defined on a nonempty finite set $\mathcal V$,  we can construct a pseudo $q$-factorization graph having $\mathcal V$ as vertex set and $\mathcal F$ as its $q$-factorization map by defining $\mathcal A$ by the requirement:
	\begin{equation*}
		(v,w)\in\mathcal A \quad\Leftrightarrow\quad V(\mathcal F(v))\otimes V(\mathcal F(w)) \ \ \text{is reducible and highest-$\ell$-weight.}
	\end{equation*}
	We denote this graph by $G(\mathcal F)$. If $\mathcal F$ is an actual $q$-factorization map over $G(\mathcal F)$ and $\bs\pi=\bs\pi_{\mathcal F}$, we also use the notation $G(\bs\pi)$ and call it the $q$-factorization graph of $\bs\pi$.  If $G(\mathcal F)$ is a fundamental factorization graph, we also use the notation $G_f(\bs\pi)$ and call it the fundamental (factorization) graph of $\bs\pi$. 
	
	Given pseudo $q$-factorization graphs $G$ and $G'$ over $\bs\pi$ and $\bs\pi'$, respectively, we denote by $G\otimes  G'$ the unique pseudo $q$-factorization graph over $\bs\pi\bs\pi'$ whose vertex set is $\mathcal V_G\, \dot{\cup}\, \mathcal V_{G'}$. Here, $\mathcal V_G$ denotes the set of vertex of the graph $G$, and so on. 
	
	As commented above, if $G$ affords the structure of a pseudo $q$-factorization graph, it contains no loops nor oriented cycles and, hence, the set $\mathcal A$ induces a partial order on $\mathcal V$ by transitive extension of the relation $h(a)<t(a)$ for all $a\in\mathcal A$. We say $G$ is totally ordered if this order is linear.

	Let $P$ be a representation theoretic property, i.e., a property assignable to modules such as being prime or real. We shall interpret $P$ as a (extrinsic) graphical property as follows. Suppose $H$ is a subgraph of a pseudo $q$-factorization graph $G$. We shall say $H$ satisfies $P$ if $V(\pi_H)$ satisfies $P$. For instance, we shall say $H$ is real if $V(\pi_H)$ is real. 
	Suppose $\mathcal G=G_1,\dots, G_l$ is a multicut of $G$, i.e., the family $G_k, 1\le k\le l$, is a family of subgraphs of $G$ with disjoint vertex sets and whose union is $\mathcal V_G$.  Given a graphical property $P$, we shall say $\mathcal G$ has the property $P$ if $G_k$ has property $P$ for all $k$ (when regarded as subgraphs of $G$).

	By a morphism from a graph $G=(\mathcal V,\mathcal A)$ to a graph $G'=(\mathcal V',\mathcal A')$ we mean a map $f:\mathcal V\to \mathcal V'$ such that
	\begin{equation*}
		\big(f(t(a)),f(h(a))\big)\in\mathcal A' \quad\text{for all}\quad a\in\mathcal A. 
	\end{equation*}
	If $G$ and $G'$ are pseudo $q$-factorization graphs, we further require that
	\begin{equation}
		\mathcal F(v) = \mathcal F'(f(v)) \quad\text{for all}\quad v\in V. 
	\end{equation}
	We use the notation $f:G\to G'$ to indicate that $f$ is a morphism from $G$ to $G'$. If $f$ is a morphism, it induces a map $\mathcal A\to\mathcal A', a\mapsto \big(f(t(a)),f(h(a))\big)$, which we also denote by $f$. 
	We say $f$ is a full morphism if 
	\begin{equation*}
		f(\mathcal A) = \mathcal A'_{G'_{f(\mathcal V)}}.
	\end{equation*}
	If $f$, as well as the induced map on arrows, are bijective, we say $f$ is an isomorphism. If an ismorphism exists, we write $G\cong G'$. 
	
	\subsection{Quochain Decompositions}\label{ss:graphmorp}
	Let us recall the notion of quochains in the sense of \cite{bms}. Let $\mathcal G=G_1,\dots, G_l$ be a multicut of a graph $G$ and set
	\begin{equation}
		\bar G_k = G_{k+1}\otimes\cdots\otimes G_l  \quad\text{for}\quad 0\le k\le l. 
	\end{equation}
	Note 
	\begin{equation}\label{e:whyquochain}
		G_{k}\triangleleft \bar G_{k-1}, \quad  \bar G_k = \bar G_{k-1}\setminus G_k \quad\text{for all}\quad 0< k\le l,
	\end{equation}
	and the sequence $\bar G_0,\dots,\bar G_l$ is a proper descending chain of subgraphs:
	\begin{equation}\label{e:defquochain}
		\emptyset=\bar G_l \triangleleft\bar G_{l-1}\triangleleft\cdots\triangleleft \bar G_1\triangleleft\bar G_0 = G.
	\end{equation}
	We shall refer to this chain as the quochain associated to $\mathcal G$. By abuse of language, we shall often refer to $\mathcal G$ as a quochain as well. 
	
	Suppose $P$ is a ``graphical property'' (including extrinsic properties), i.e., a property assignable to subgraphs of a graph, such as being connected, totally ordered, real, prime, etc.. We shall say the multicut $\mathcal G$ determines a $P$-quochain (or that $\mathcal G$ is a $P$-quochain by abuse of language) if $G_k$ has the property $P$ when regarded as a subgraph of $\bar G_{k-1}$ for all $1\le k\le l$. 
	
	We now introduce two notions which will play a prominent role in this paper, starting with that of isomorphic quochains. 
	Two quochains $G_1,\dots, G_l$ and $G_1',\dots,G'_{l'}$ will be said to be isomorphic if $l'=l$ and there exists $\sigma\in\mathcal S_l$ such that
	\begin{equation*}
		G'_k\cong G_{\sigma(k)} \quad\text{for all}\quad 1\le k\le l.
	\end{equation*}
	We shall say $G$ has a unique $P$-decomposition if $G$ admits a $P$-quochain and every two $P$-quochains in $G$ are isomorphic. 
	For instance, if $P$ is ``being a maximal connected subgraph of $G$'', then $G$ admits a $P$-quochain and any $P$-quochain corresponds to an enumeration of the connected components of $G$. Thus, $G$ has a unique $P$-decomposition. 
	
	We shall be concerned with the case that $P$ is ``being a maximal totally ordered subgraph''. For short, we shall refer to quochains with such property as mtos-quochains and to the associated decompositions as  mtos-decompositions. Evidently, every pseudo $q$-factorization graph $G$ admits an mtos-quochain.

	\subsection{Snake Modules and Totally Ordered Pseudo $q$-Factorization Graphs}\label{ss:snakemods}
	Let us recall the definition of snake modules for type $A$. Given $(i_k,a_k)\in I\times\mathbb Z, k\in\{1,2\}$, it is said that the ordered pair $((i_1,a_1),(i_2,a_2))$ is in snake position if 
	\begin{equation*}
		a_2-a_1 \in \min\mathscr R_{i_1,i_2} + 2\mathbb Z_{\ge 0} =  d(i_1,i_2) + 2\mathbb Z_{>0}. 
	\end{equation*}
	If $a_2-a_1 \in \mathscr R_{i_1,i_2}$, then it is said the pair is in prime snake position. More generally, if $\bs i = (i_1,\dots, i_l)\in I^l$ and $\bs a=(a_1,\dots a_l)\in\mathbb Z^l$, it is said that $(\bs i,\bs a)$ is a (prime) snake if every pair $((i_k,a_k),(i_{k+1},a_{k+1})), 1\le k<l$, is in (prime) snake position. Given $(\bs i,\bs a)\in I^l\times\mathbb Z^l$, set
	\begin{equation*}
		\bs\omega_{\bs i,\bs a} = \prod_{k=1}^l \bs\omega_{i_k,a_k}. 
	\end{equation*}
	If $(\bs i,\bs a)$ is a (prime) snake, the module $V(\bs\omega_{\bs i,\bs a})$ is called a (prime) snake module. The first of our main results gives an alternate perspective for the definition of prime snake modules. 
	
	\begin{thm}\label{t:ps<=>toG}
		Let $\bs\pi\in\mathcal P^+$. Then, $V(\bs\pi)$ is a prime snake module if and only if every pseudo $q$-factorization graph over $\bs\pi$ is totally ordered.  \endd
	\end{thm}
	
	This theorem will be proved as an application of a result about the concept of fusing vertices of a pseudo $q$-factorization graph which we introduce in \Cref{ss:fuse}. 
	
	\subsection{Snake Support and Prime Factorization} \label{ss:sspf}
	Given $\bs\pi\in\mathcal P^+$, let $\mathcal F$ be the pseudo $q$-factoriza\-tion map associated to $G:=G_f(\bs\pi)$ and consider
	\begin{equation}
		\bbs\pi = \prod_{\bs\omega\in\mathcal F(\mathcal V)}\bs\omega.  
	\end{equation}
	By definition, $\bbs\pi$ has exactly one copy of each fundamental factor of $\bs\pi$. We shall say $\bs\pi$ has snake support if $V(\bbs\pi)$ is a prime snake module. As usual, by abuse of language, we shall also say $G$ has snake support. We are ready to state the second of our main results.
	
	\begin{thm}\label{t:orddecss}
		If $\bs\pi$ has snake support,  then $G=G_f(\bs\pi)$ has a unique mtos-decomposition. Moreover, if $\mathcal G=G_1,\dots,G_l$ is an mtos-quochain, then
		\begin{equation*}
			V(\bs\pi)\cong V(\bs\pi_{G_1})\otimes \cdots\otimes V(\bs\pi_{G_l})
		\end{equation*} 
		is the unique prime factorization of $V(\bs\pi)$. \endd
	\end{thm}
	
	The following lemma will be proved in \Cref{ss:monoc}.

	\begin{lem}\label{l:moncto}
		If $\bs\pi\in\mathcal P^+_i$ for some $i\in I$, the connected components of $G_f(\bbs\pi)$ are totally ordered.\endd
	\end{lem}

	As a consequence, if $\bs\pi\in\mathcal P^+_i$ for some $i\in I$, the prime factorization of $V(\bs\pi)$ is given by \Cref{t:orddecss}. The prime factorization of such modules was described in \cite[Theorem 2]{bc:hokr} in terms of the concept of $(i,n)$-segments. Thus, \Cref{t:orddecss} can be seen as a generalization of \cite[Theorem 2]{bc:hokr}. 
	We now establish the dictionary between the two languages. 
	
	Given $\bs k=(k_1,\dots,k_l)\in \mathbb Z^l$, set 
	\begin{equation*}
		\bs k^\pm = (k_1^\pm,\dots, k_{l-1}^\pm) \quad\text{with}\quad k_s^\pm = k_{s+1}\pm k_s. 
	\end{equation*}
	An $(i,n)$-segment of length $l$ was defined in \cite{bc:hokr} as a sequence $\bs k=(k_1,k_2,\dots,k_l)\in\mathbb Z^l$ such that 
	\begin{equation}
		k_s^- \in \mathscr R_{i,i} \quad\text{for all}\quad 1\le s<l.
	\end{equation}
	We shall denote by $S_{i,n,l}$ the set of $(i,n)$-segments of length $l$ and set
	\begin{equation*}
		S_{i,n}=\bigcup_{l\in\mathbb Z_{>0}} S_{i,n,l}.
	\end{equation*}
	Given $\bs k\in S_{i,n}$, we let $\ell(\bs k)$ denote its length, i.e., $\ell(\bs k)=l$ if and only if $\bs k\in S_{i,n,l}$.

	Given $\bs k\in\mathbb Z^l, i\in I,$ and $a\in\mathbb F^\times$, set
	\begin{equation*}
		\bs\varpi_{i,\bs k,a} = \prod_{s=1}^l \bs\omega_{i,a+k_s}\in\mathcal P_i^+.
	\end{equation*}
	Note that, if $\bs k\in S_{i,n}$ and $G=G_f(\bs\varpi_{i,\bs k,a})$, then
	\begin{equation*}
		(\bs\omega_{i,a+k_{s+1}},\bs\omega_{i,a+k_{s}})\in\mathcal A_G \quad\text{for all}\quad 1 \le s<\ell(\bs k).
	\end{equation*}
	Thus, the fundamental factorization graph associated to the Drinfeld polynomial of an $(i,n)$-segment is totally ordered. 
	In other words, the concept of $(i,n)$-segments is equivalent to that of prime snakes supported at a single node of the Dynkin diagram. In particular, it follows from \Cref{t:ps<=>toG}  that any pseudo $q$-factorization graph over $\bs\varpi_{i,\bs k,a}$ is totally ordered. In particular, the corresponding simple $U_q(\tlie g)$-module is prime and strongly real in the sense of \cite{bms}. Continuing along the lines of the proof of \Cref{t:ps<=>toG}, we shall prove the following in \Cref{ss:monoc}. 
	
	\begin{thm}\label{p:equivinseg}
		Let $\bs\pi\in\mathcal P^+_i$ for some $i\in I$. The following conditions are equivalent:
		\begin{enumerate}[(i)]
			\item $\bs\pi = \bs\varpi_{i,\bs k,a}$ for some $\bs k\in S_{i,n}, a\in\mathbb Z$.
			\item $V(\bs\pi)$ is a prime snake module.
			\item $G_f(\bs\pi)$ is totally ordered.
			\item $G(\bs\pi)$ is totally ordered. 
			\item $V(\bs\pi)$ is prime. \endd
		\end{enumerate}
	\end{thm}
	
	It follows from the results of this \Cref{ss:sspf} that,  for monochromatic fundamental factorization graphs, the concept of mtos-quochain  is equivalent to that of maximal prime snake quochains or, equivalently, quochains of maximal $(i,n)$-segments.

	\subsection{Further Comments}\label{ss:evmcg}
	\Cref{ss:sspf} implies that a ``having snake support'' is a sufficient condition for the prime factorization to arise from mtos-decompositions. Are there other classes of modules with the same property? Is there $\bs\pi$ such that $G(\bbs\pi)$ is totally ordered but the prime factorization does not arise from studying mtos-decompositions? These are questions we find worth of further investigation. 
	
	Let us end our discussion by making some comments related to these questions according to the number of vertices in $G(\bs\pi)$. Evidently, if the number of vertices is at most two, then the prime factorization arises from mtos-decompositions. In the case of $3$ vertices, the prime factorization follows from the results of \cite{ms:3tree}, which we now recall. If $G(\bs\pi)$ has more than one connected component, then all of them are totally ordered trees and, hence, prime and strongly real. In that case, if $G_k, 1\le k\le l$ with $2\le l\le 3$ is an enumeration of the connected components of $G$, then
	\begin{equation*}
		V(\bs\pi) \cong \bigotimes_{k=1}^l V(\bs\pi_{G_k})
	\end{equation*}
	is the prime factorization. Thus, henceforth assume $G$ is connected. If $G$ is totally ordered, then $G$ is prime by \cite{ms:to} and there is nothing further to be done. This happens exactly when $G$ is either a triangle or a tree which is a directed path. If $G$ is not totally ordered, then it must be of the form
	\begin{equation}\label{e:3valg}
		\begin{tikzcd}
			\stackrel{r_1}{i_1} & \arrow[swap,l,"m_1"]  \stackrel{r}{i} \arrow[r,"m_2"] & \stackrel{r_2}{i_2} 
		\end{tikzcd} \qquad\text{or}\qquad 
		\begin{tikzcd}
			\stackrel{r_1}{i_1} \arrow[r,"m_1"] &   \stackrel{r}{i}  & \arrow[swap,l,"m_2"] \stackrel{r_2}{i_2} 
		\end{tikzcd}
	\end{equation}
	The corresponding prime factorization is described by the following theorem. We need to review further notation for the statement. If  $J\subseteq I$ is a connected subdiagram and $[i,j]\subseteq J$, let $\mathscr R_{i,j,J}^{r,s}$ be defined as in \eqref{e:defredset'} with $J$ in place of $I$. We also let $w_0^J$ denote the non trivial diagram automorphism of $J$. 
	
	\begin{thm}[{\cite[Theorem 2.4.6]{ms:3tree}}]\label{t:3lineprime}
		Assume $\lie g$ is of type $A$ and let $G=G(\bs\pi)$ be an alternating line as above. For $j=1,2$, let also $I_j\subseteq I$ be the minimal connected subdiagram containing $[i,i_j]$ such that $m_j\in\mathscr R_{i,i_j,I_j}^{r,r_j}$ and let $j'$ be such that $\{j,j'\}=\{1,2\}$. Then, $G$ is not prime if and only if there exists $j\in\{1,2\}$ such that
		\begin{equation*}\label{2gencondm}
			i_{j'}\in I_j, \qquad m_{j'}\in{\mathscr{R}_{i,i_{j'},I_j}^{r,r_{j'}}}, \qquad m_{j'}-m_j+\check h_{I_j}\in\mathscr R_{w_0^{I_j}(i_j),i_{j'},I_j}^{r_j,r_{j'}},
		\end{equation*}
		and
		\begin{equation}\label{2exracondm}
			m_j+r_j \le m_{j'} + r_{j'} +d(i_1,i_2).
		\end{equation}
		In that case, $V(\bs\pi)\cong V(\bs\omega)\otimes V(\bs\pi\bs\omega^{-1})$, where $\bs\omega$ is the $q$-factor corresponding to such $j$.\qed
	\end{thm}
	
	Note that \Cref{t:3lineprime} characterizes the  $q$-factorization graphs afforded by three-vertex alternating lines which are prime. Each such example is an example of a module whose prime factorization does not arise from studying maximal totally ordered subgraphs. For instance, this is the case for $\bs\pi = \bs\omega_{1,3}\,\bs\omega_{2,0}\,\bs\omega_{3,3}$.  
	
	In the case that $i_1=i_2=i$ in \Cref{t:3lineprime}, \Cref{l:moncto} implies the prime factorization given in the last line of this theorem must coincide with that described in \Cref{t:orddecss}. In the spirit of establishing the dictionary between the results of this paper and the previous literature, we dedicate the remainder of this section to make this checking explicit. 
	
	Let $a\in\mathbb Z$ be such that the middle vertex of $G$ is $\bs\omega_{i,a,r}$, and, without loss of generality, assume this vertex is a sink. In that case, the other vertices are
	\begin{equation*}
		\bs\omega_{i,a+m_j,r_j} \quad j\in\{1,2\}. 
	\end{equation*}
	Let also
	\begin{equation*}
		G_j = G_f(\bs\omega_{i,a,r}\,\bs\omega_{i,a+m_j,r_j})
	\end{equation*}
	and note 
	\begin{equation}
		r_j\le r_{j'} \quad\Rightarrow\quad G_{j'} \text{ is a maximal totally ordered subgraph of $G_f(\bs\pi)$.}
	\end{equation}
	In that case, \Cref{t:orddecss} implies 
	\begin{equation}
		V(\bs\pi)\cong V(\bs\omega_{i,a,r}\,\bs\omega_{i,a+m_{j'},r_{j'}})\otimes V(\bs\omega_{\bs\omega_{i,+m_j,r_j}})
	\end{equation}
	is the prime factorization. Note that, if $r_j=r_{j'}$, both $G_j$ and $G_{j'}$ give rise to mtos-decompositions. Let us check \Cref{t:3lineprime} agrees with this. Note we have
	\begin{equation*}
		m_j = r+r_j -2p_j \quad\text{with}\quad -d(i,\partial I)\le p_j<0, \quad I_j = [i+p_j,i-p_j],  \quad\text{and}\quad w_0^{I_j}(i_j)=i.
	\end{equation*}
	The fact that $p_j<0$ arises from the assumption we are working with the actual $q$-factorization. 
	In particular, the condition $i_{j'}\in I_j$ holds for both choices of $j$ and the condition $m_{j'}\in{\mathscr{R}_{i,i_{j'},I_j}^{r,r_{j'}}}$ is equivalent to $I_{j'}\subseteq I_j$, i.e., $p_j\le p_{j'}$, which certainly holds for at least one choice of $j$. Thus, choose such $j$. The third condition in the theorem can be rephrased as
	\begin{equation}\label{e:3rdcnp}
		r_j+r_{j'} - 2(r_j +p_{j'}-1)\in\mathscr R_{i,i,I_j}^{r_j,r_{j'}},
	\end{equation}
	while \cite[(4.3.17)]{ms:3tree} says \eqref{2exracondm} is equivalent to $r_j\le r_{j'}$, which implies \eqref{e:3rdcnp} in this case. 
	Thus, in order to check \Cref{t:3lineprime} agrees with  \Cref{t:orddecss}, it remains to check that the assumption $|m_j-m_{j'}|\notin\mathscr R_{i,i}^{r_j,r_j'}$, together with the above choice of $j$, implies $r_j\le r_{j'}$. 
	
	For doing this, begin by noting that the choice of $j$ by itself implies $m_{j'}-m_j \notin\mathscr R_{i,i}^{r_j,r_j'}$. Indeed,
	\begin{equation*}
		m_{j'}-m_j = r_j+r_{j'} - 2(p_{j'}-p_j+r_j)
	\end{equation*}
	and, since $p_j\le p_{j'}$, we have $p_{j'}-p_j+r_j\ge r_j$. Thus, it remains to check 
	\begin{equation*}
		m_j-m_{j'}\notin\mathscr R_{i,i}^{r_j,r_j'} \quad\Rightarrow\quad r_j\le r_{j'}.
	\end{equation*}
	Begin by noting that 
	\begin{equation*}
		m_j-m_{j'} = r_j+r_{j'} - 2(p_j-p_{j'} + r_{j'})
	\end{equation*}
	and
	\begin{equation*}
		p_{j'}-p_j = p_{j'} + |p_j| < d(i,\partial I) 
	\end{equation*}
	since $p_{j'}<0$ and $|p_j| = -p_j \le d(i,\partial I)$. Therefore,
	\begin{equation*}
		p_j-p_{j'}+r_{j'} > p_j-p_{j'} > - d(i,\partial I). 
	\end{equation*}
	Moreover, since $p_j\le p_{j'}$, it follows that 
	\begin{equation*}
		m_j-m_{j'}\notin\mathscr R_{i,i}^{r_j,r_j'} \quad\Leftrightarrow\quad p_j-p_{j'}+r_{j'}\ge r_j.
	\end{equation*}
	But the latter condition coincides with \eqref{2exracondm} since $i_1=i_2$, thus completing the checking.

	\section{The Proofs}\label{s:pfs}
	
	\subsection{Collected Technical Results}
	In this section, for the readers convenience, we collect a few technical results we shall need from the existing literature.
	
	\begin{lem}[{\cite[Lemma 5.1.2]{ms:to}}]\label{l:lineps}
		Let $N\in\mathbb Z_{> 0}$ and $(m_k,r_k,i_k)\in \mathbb Z_{\ge 0}\times \mathbb Z_{> 0}\times I, 1\le k\le N$. Suppose 
		\begin{equation}\label{e:lineps}
			|m_k-m_{k-1}|\in\mathscr{R}^{r_{k-1},r_k}_{i_{k-1},i_k} \qquad\text{for all}\qquad 1< k\leq N.
		\end{equation}
		\begin{enumerate}[(a)]
			\item For all $1\le k,l\le N$, there exists $p_{l,k}\in\mathbb Z$ such that $m_l-m_k=r_l+r_k+d(i_l,i_k)-2p_{l,k}$. 
			\item If $m_k>m_{k-1}$ for all $1< k\leq N$, then $p_{N,1}<\operatorname{min}\{r_1,r_N\}$, and 
			\begin{equation*}
				p_{N,1}<p_{l,k}<\operatorname{min}\{r_k,r_l\},\quad\textrm{for all}\quad 1\leq k<l\leq N,\quad\textrm{with}\quad (k,l)\neq(1,N).
			\end{equation*}
			Similarly, if  $m_k<m_{k-1}$ for all $1< k\leq N$, then $p_{1,N}<\operatorname{min}\{r_1,r_N\}$, and 
			\begin{equation*}
				p_{1,N}<p_{k,l}<\operatorname{min}\{r_k,r_l\},\quad\textrm{for all}\quad 1\leq k<l\leq N,\quad\textrm{with}\quad (k,l)\neq(1,N).
			\end{equation*}
		\end{enumerate}\qed
	\end{lem}
	
	\begin{lem}[{\cite[Lemma 5.1.3]{ms:to}}]\label{l:positiveppath}
		Assume $m_k>m_{k-1}$ for all $1< k\leq N$ in \Cref{l:lineps}. 
		\begin{enumerate}[(a)]
			\item If $p_{N,1}\ge -d([i_k,i_l],\partial I)-1$ for some  $1\leq k<l\leq N, (l,k)\ne (1,N)$, then $m_l-m_k\in\mathscr R^{r_k,r_l}_{i_k,i_l}$. In particular, this is the case if $m_N-m_1\in\mathscr R^{r_1,r_N}_{i_1,i_N}$ and $d([i_k,i_l],\partial I)\ge d([i_1,i_N],\partial I)$.
			\item If $m_N-m_1\in\mathscr R^{r_1,r_N}_{i_1,i_N,[i_1,i_N]}$, then $m_l-m_k\in\mathscr R^{r_k,r_l}_{i_k,i_l,[i_k,i_l]}$	for all $1\leq k<l\leq N$.\qed
		\end{enumerate}
	\end{lem}
	
	We now collect a few known facts about tensor products of highest-$\ell$-weight modules. 
	The following are well-known.
	
	\begin{prop}\label{p:sinter}
		Let $\bs{\pi}, \bs{\varpi}\in\mathcal{P}^+$. Then, $V(\bs{\pi})\otimes V(\bs{\varpi})$ is simple if and only if $V(\bs{\varpi})\otimes V(\bs{\pi})$ is simple and, in that case, $V(\bs{\pi})\otimes V(\bs{\varpi}) \cong V(\bs\pi\bs\varpi)\cong V(\bs{\varpi})\otimes V(\bs{\pi})$. 	\hfil\qed
	\end{prop}

	\begin{prop}\label{p:vnvstar}
		Let $\bs{\pi}, \bs\varpi\in\mathcal{P}^+$. Then, $V(\bs{\pi})\otimes V(\bs\varpi)$ is simple if and only if both $V(\bs{\pi})\otimes V(\bs\varpi)$ and $V(\bs\varpi)\otimes V(\bs{\pi})$ are highest-$\ell$-weight. 	\hfil\qed
	\end{prop}
	
	As in \cite{ms:to,ms:3tree,bms}, the following theorem plays a crucial role in the main proofs of this paper. For comments on its proof, see \cite[Remark 4.1.7]{ms:to}.
	
	\begin{thm}\label{t:cyc}
		Let $S_1,\cdots, S_m\in\mathcal C$ be simple and assume $S_i$ is real either for all $i>2$ or for all $i<m-1$. If $S_i\otimes S_j$ is highest-$\ell$-weight for all $1\leq i< j\leq m$, then $S_1\otimes\cdots\otimes S_m$ is highest-$\ell$-weight. 	\hfil\qed
	\end{thm}
	
	\begin{cor}\label{c:cyc}
		Let $S_1,\cdots, S_m\in\mathcal C$ be simple and assume $S_i$ is real either for all $i>2$ or for all $i<m-1$. Then, $S_1\otimes\cdots\otimes S_m$ is simple if, and only if, $S_i\otimes S_j$ is simple for all $1\leq i< j\leq m$.  	\hfil\qed
	\end{cor}

	The following played a crucial role in the proof of \cite[Corollary 3.5.4]{bms}, which shows prime snake modules are strongly real. It will also play a relevant role in the proof of \Cref{t:orddecss}. 
	
	\begin{prop}[{\cite[Proposition 4.1.3(ii)]{naoi:Tsys}}]\label{p:testworksforsnakes}
		Suppose $V(\bs\pi)$ is a prime snake module and that $\bs\omega$ divides $\bs\pi$. Then, $V(\bs\pi)\otimes V(\bs\omega)$ is simple.  \qed
	\end{prop}
	
	Finally, let us collect some results about $\ell$-weights, starting with the well-known facts that
	\begin{equation}\label{e:lwprod}
		\wt_\ell V(\bs\pi_1\bs\pi_2)\subseteq \wt_\ell V(\bs\pi_1)\otimes V(\bs\pi_2) = (\wt_\ell V(\bs\pi_1)) (\wt_\ell V(\bs\pi_2)). 
	\end{equation} 
	As usual, given $\bs\omega\in\mathcal P, i\in I, a\in\mathbb Z$, we shall say $\bs\omega_{i,a}^{\pm 1}$ occurs in $\bs\omega$ if it appears as a factor in a reduced expression of $\bs\omega$ in terms of fundamental $\ell$-weights. More generally, given $\bs\omega,\bs\varpi\in\mathcal P$, we shall say $\bs\varpi$ occurs in $\bs\omega$ if the reduced expression for $\bs\omega$ contains the one for $\bs\varpi$.  The following is also well known:
	\begin{equation}\label{e:lwbounds}
		\bs\omega\in\wt_\ell V(\bs\omega_{i,a})\setminus\{\bs\omega_{i,a}\} \text{ and } \bs\omega_{j,b}^{\pm 1} \text{ occurs in } \bs\omega \quad\Rightarrow\quad a+1\leq  b\le a+h,
	\end{equation}
	and, moreover,
	\begin{equation*}
		b = a+h \quad\Leftrightarrow\quad \bs\omega = (\bs\omega_{i^*,a+h})^{-1} = \,({}^*\bs\omega_{i,a})^{-1}. 
	\end{equation*}
	Given $\bs\pi_1,\bs\pi_2\in\mathcal P^+$, let us write $\bs\pi_1 \gtrdot \bs\pi_2$ if 
	\begin{equation}
		\bs\omega_{i_j,a_j} \text{ occurs in } \bs\pi_j, \ j\in\{1,2\} \quad\Rightarrow\quad a_1>a_2. 
	\end{equation}
	Note that, since fundamental modules are real, an application of \eqref{e:krhwtp} and \Cref{t:cyc} gives:
	\begin{equation}\label{e:rmhwtp}
		\bs\pi_1 \gtrdot \bs\pi_2 \quad\Rightarrow\quad V(\bs\pi_1)\otimes V(\bs\pi_2) \text{ is highest-$\ell$-weight.}
	\end{equation}
	The following lemma follows from an immediate application of \eqref{e:lwprod} and \eqref{e:lwbounds}. 
	
	\begin{lem}\label{l:extocc}
		Suppose $\bs\pi_1,\bs\pi_2\in\mathcal P^+$ satisfy $\bs\pi_1 \gtrdot \bs\pi_2$ and let $\bs\omega\in\wt_\ell V(\bs\pi_1)\otimes V(\bs\pi_2)$. 
		\begin{enumerate}[(a)]
			\item 	If $\bs\pi_2$ occurs in $\bs\omega$, then $\bs\omega = \bs\pi_2 \bs\omega'$ for some $\bs\omega'\in\wt_\ell V(\bs\pi_1)$. 
			\item If $({}^*\bs\pi_1)^{-1}$ occurs in $\bs\omega$, then $\bs\omega = ({}^*\bs\pi_1)^{-1} \bs\omega'$ for some $\bs\omega'\in\wt_\ell V(\bs\pi_2)$. \qed
		\end{enumerate}
	\end{lem}
	
	We shall also need the following extract from the main results of \cite{my:pathB, muyou:tsystem}. Let $(\bs i,\bs a)$ be a snake and write $\bs i=(i_1,\cdots, i_l)$ and $\bs a = (a_1,\cdots, a_l)$. Given $1\leq r\leq s\leq l$, let 
	$$ \bs\omega_{\bs i,\bs a,r,s} = \prod_{j=r}^s\bs\omega_{i_j,a_j}.$$
	For convenience, set $\bs\omega_{\bs i,\bs a,1,0}=\bs 1 =\bs\omega_{\bs i,\bs a,l+1,l}$. 
	It follows from \cite[Theorem 6.1]{my:pathB} that 
	\begin{equation}\label{e:my:tsyslw}
		\bs\omega\in \wt_\ell V(\bs\omega_{\bs i,\bs a,r,s}) \quad\Rightarrow\quad 
		\bs\omega_{\bs i,\bs a,1,r-1}\bs\omega ({}^*\bs\omega_{\bs i,\bs a,s+1,l})^{-1}\in \wt_\ell V(\bs\omega_{\bs i,\bs a}).
	\end{equation}
	Moreover, by \cite[Theorem 4.1]{muyou:tsystem},
	\begin{equation}\label{e:my:tsysred}
		\text{$l>1$ and $(\bs i,\bs a)$ is a prime snake} \quad\Rightarrow\quad  V(\bs\omega_{\bs i,\bs a,1,l-1})\otimes V(\bs\omega_{\bs i,\bs a,2,l}) \text{ is reducible}
	\end{equation} 
	and there exists $\bs\omega\in \wt_\ell  V(\bs\omega_{\bs i,\bs a,1,l-1}), \bs\omega' \in \wt_\ell  V(\bs\omega_{\bs i,\bs a,2,l})$ such that
	\begin{equation}\label{e:my:tsysnlw}
		\bs\omega\bs\omega'\notin \wt_\ell  V(\bs\omega_{\bs i,\bs a,1,l-1}\,\bs\omega_{\bs i,\bs a,2,l}). 
	\end{equation}

	\subsection{Vertex Fusion and Linear Order Preservation}\label{ss:fuse}
	In this section, we initiate an analyses about how a choice of pseudo $q$-factorization graph may be more convenient than others depending on the type of result one is trying to prove about the underlying simple module. As an application of this initial analyses, we prove \Cref{t:ps<=>toG}.

	Let $a,b\in\mathbb Z, r,s\in\mathbb Z_{>0}$, and $i\in I$. If $|a-b|\in\mathscr R_i^{r,s}$, it is said that $\bs\omega_{i,a,r}$ and $\bs\omega_{i,b,s}$ are in special position. In that case, it is well-known that there exist unique $c,d\in\mathbb Z, k,l\in\mathbb Z_{\ge 0}$ such that $\bs\omega_{i,c,k}\,\bs\omega_{i,d,l}$ is the $q$-factorization of $\bs\omega_{i,a,r}\,\bs\omega_{i,b,s}$ and $\bs\omega_{i,c,k}$ is divisible by both $\bs\omega_{i,a,r}$ and $\bs\omega_{i,b,s}$. Let us introduce the following terminology: we shall say that $\bs\omega_{i,c,k}$ and $\bs\omega_{i,d,l}$ arise from (or are the result of) the fusion of the KR type polynomials  $\bs\omega_{i,a,r}$ and $\bs\omega_{i,b,s}$. We say the fusion is pure if $k=r+s$ or, equivalently, if $\bs\omega_{i,a,r}\,\bs\omega_{i,b,s}\in\mathcal{KR}$. 
	
	Suppose  $G=(\mathcal V,\mathcal A,\mathcal F)$ is a pseudo $q$-factorization graph over $\bs\pi\in\mathcal P^+$ and that $v,w\in\mathcal V$ are vertices such that $\mathcal F(v) = \bs\omega_{i,a,r}$ and $\mathcal F(w)=\bs\omega_{i,b,s}$, as in the previous paragraph. We can then consider the following pseudo $q$-factorization graph $G'=(\mathcal V',\mathcal A',\mathcal F')$. Recall that, once we have defined $\mathcal V'$ and any map $\mathcal F':\mathcal V'\to\mathcal{KR}$, there exists unique $\mathcal A'$ turning $G'$ into a pseudo $q$-factorization graph. 
	If the fusion is not pure, we let $\mathcal V'=\mathcal V$, and set 
	\begin{equation*}
		\mathcal F'(v) = \bs\omega_{i,c,k}, \quad \mathcal F'(w) = \bs\omega_{i,d,l}, \quad\text{and}\quad \mathcal F'(u)=\mathcal F(u) \text{ for all } u\in\mathcal V\setminus\{v,w\}. 
	\end{equation*}
	The roles of $v$ and $w$ are clearly interchangeable above, i.e., the obtained graphs are isomorphic. In the case the fusion is pure, we let $\mathcal V'=\mathcal V\setminus\{w\}$ and set
	\begin{equation*}
		\mathcal F'(v) = \bs\omega_{i,a,r}\bs\omega_{i,b,s} \quad\text{and}\quad \mathcal F'(u)=\mathcal F(u) \text{ for all } u\in\mathcal V'\setminus\{v\}. 
	\end{equation*}
	Again, the roles of $v$ and $w$ are interchangeable. Evidently, $\bs\pi_{G'}=\bs\pi$. We shall say $G'$ is obtained from $G$ by fusing two vertices and we indicate that by the notation $G'\prec G$. If the fusion is pure, we write $G'< G$. Both $\prec$ and $<$ can be extended by reflexivity and transitivity to equip the set ${\rm Graph}(\bs\pi)$ of pseudo $q$-factorization graphs over $\bs\pi$ with two distinct partial orders. Evidently,
	\begin{equation}
		G'\le G \quad\Rightarrow\quad G'\preceq G,
	\end{equation} 
	but the converse may not be true. Also
	\begin{equation}
		G(\bs\pi)\preceq G \ \ \text{and}\ \ G\le G_f(\bs\pi) \quad\text{for every}\quad G\in {\rm Graph}(\bs\pi).
	\end{equation}
	In particular, $G_f(\bs\pi)$ is the maximum for both partial orders.

	We now address the following question. Suppose $G'\le G$ and that $G$ is totally ordered. Under which additional assumptions can we conclude $G'$ is also totally ordered? To shorten the writing, we go back to abusing of terminology by saying that $\bs\omega\in\mathcal KR$ is a vertex of a given graph, instead of saying $\bs\omega=\mathcal F(v)$ for some vertex $v$. 
	
	It suffices to address the question in the case that $G'$ is obtained from $G$ by fusing only one pair of vertices. Let $\bs\omega$ and $\bs\varpi$ be the fused vertices and assume $(\bs\omega,\bs\varpi)\in \mathcal A$. In particular,
	\begin{equation}\label{e:prodiskr}
		\bs\omega\bs\varpi\in\mathcal{KR}
	\end{equation}
	is a vertex in $G'$. All the remaining vertices of $G$ remain vertices in $G'$. 
	Thus, given $\bs\mu\in\mathcal V\setminus\{\bs\omega,\bs\varpi\}$, we need to  understand under which conditions $\bs\mu$ is comparable to $\bs\omega\bs\varpi$ with respect to the partial order of $\mathcal V'$. Without loss of generality, we may assume $\bs\mu$ is adjacent to either $\bs\omega$ or $\bs\varpi$.

	Write $\bs\omega = \bs\omega_{i,a,r}$ and $\bs\varpi=\bs\omega_{j,b,s}$. Then, the assumption $(\bs\omega,\bs\varpi)\in \mathcal A$ together with \eqref{e:prodiskr} implies
	\begin{equation}\label{e:pfcentrel}
		a = b +r+s \quad\text{and}\quad \bs\omega\bs\varpi = \bs\omega_{i,a-s,r+s}.
	\end{equation}
	Let also $\bs\mu = \bs\omega_{j,c,t}$. 
	Let us check
	\begin{equation*}
		(\bs\mu,\bs\omega)\in\mathcal A \quad\Rightarrow\quad (\bs\mu,\bs\omega\bs\varpi)\in\mathcal A'.
	\end{equation*}	
	Indeed, the assumption $(\bs\mu,\bs\omega)\in\mathcal A$ means
	\begin{equation*}
		c-a = r+t + d(i,j)-2p \quad\text{with}\quad -d([i,j],\partial I)\le p<\min\{r,t\},
	\end{equation*}
	and, therefore,
	\begin{equation*}
		c-(a-s) = (r+s)+t + d(i,j)-2p \quad\text{and}\quad -d([i,j],\partial I)\le p<\min\{r+s,t\}.
	\end{equation*}
	Similarly, one checks
	\begin{equation*}
		(\bs\varpi,\bs\mu)\in\mathcal A \quad\Rightarrow\quad (\bs\omega\bs\varpi,\bs\mu)\in\mathcal A'.
	\end{equation*}	
	Thus, $G'$ may fail to be totally ordered only if 
	\begin{equation}\label{e:middlemu}
		\exists\ \bs\mu\in\mathcal V \quad\text{such that}\quad \bs\varpi<\bs\mu<\bs\omega.
	\end{equation}
	Here, $<$ is the partial order on $\mathcal V$ induced by $\mathcal A$. Up to arrow duality, we may assume $\bs\mu$ is adjacent to $\bs\omega$, so $(\bs\omega,\bs\mu)\in\mathcal A$ and, hence,
	\begin{equation}\label{e:acrel}
		a-c = r+t + d(i,j)-2p \quad\text{with}\quad -d([i,j],\partial I)\le p<\min\{r,t\}.
	\end{equation}
	On the other hand,  \Cref{l:lineps}(b) implies
	\begin{equation}\label{e:bcrel}
		c-b = s+t+d(i,j)-2k \quad\text{for some}\quad k<\min\{s,t\}.
	\end{equation}
	Using \eqref{e:pfcentrel} and \eqref{e:acrel}, we also have
	\begin{equation*}
		c-b = c-(a-r-s) = (c-a) + r+ s = s-t-d(i,j)+2p,
	\end{equation*}
	and, hence,
	\begin{equation}
		t+d(i,j) = k+p.
	\end{equation}
	The upper bounds on \eqref{e:acrel} and \eqref{e:bcrel} then imply
	\begin{equation*}
		t+d(i,j) - k < t \quad\text{and}\quad t+d(i,j) - p < t,
	\end{equation*}
	or, equivalently,
	\begin{equation}\label{e:dupb}
		d(i,j)<\min\{p,k\}.
	\end{equation}
	If $\min\{r,s\}=1$, it follows that $\min\{p,k\}\le 0$ and, hence, \eqref{e:dupb} is not satisfied, i.e., there is no $\bs\mu$ as in \eqref{e:middlemu}. This proves:
	
	\begin{prop}\label{e:topreserved}
		Let $G$ be a totally ordered pseudo $q$-factorization graph. If $G'$ is obtained from $G$ by a sequence of pure fusions such that in every step at least one of the vertices being fused is fundamental, then $G'$ is also totally ordered. \qed  
	\end{prop}
	
	\begin{cor}\label{c:topreserved}
		Let $G$ be a fundamental factorization graph. If $G$ is totally ordered, then any pseudo $q$-factorization graph over $\bs\pi_G$ is totally ordered.\qed
	\end{cor}

	Let us turn to the proof of \Cref{t:ps<=>toG}.  Given $(\bs i,\bs a)\in I^l\times\mathbb Z^l$, it is well-known and easily checked that the element obtained by removing $(i_k,a_k)$ for some $1\le k\le l$ is also a snake.  This implies
	\begin{equation}\label{e:subsnake}
		(\bs i,\bs a) \text{ is a snake and } \bs\omega \text{ divides } \bs\omega_{\bs i,\bs a} \quad\Rightarrow\quad V(\bs\omega) \text{ is a snake module}.
	\end{equation}
	The following is easily checked directly from the definitions of snakes and factorization graphs.
	
	\begin{lem}\label{l:equivdefpsnake}
		If $(\bs i,\bs a)$ is a snake, the following are equivalent.
		\begin{enumerate}[(i)]
			\item $(\bs i,\bs a)$ is a prime snake.
			\item $G_f(\bs\omega_{\bs i,\bs a})$ is connected.
			\item $G_f(\bs\omega_{\bs i,\bs a})$ is totally ordered.\qed
		\end{enumerate}
	\end{lem}
	
	\begin{proof}[Proof of \Cref{t:ps<=>toG}]
		If $V(\bs\pi)$ is a prime snake module, \Cref{l:equivdefpsnake} implies $G_f(\bs\pi)$ is totally ordered while \Cref{c:topreserved} implies every pseudo $q$-factorization graph over $\bs\pi$ is totally ordered. Conversely, in particular $G=G_f(\bs\pi)$ is totally ordered. In light of \Cref{l:equivdefpsnake}, we need to show $\bs\pi=\bs\omega_{\bs i,\bs a}$ for some snake $(\bs i,\bs a)$. Let $\bs\omega_{i_k,a_k}, 1\le k\le l$, be the fundamental factors of $\bs\pi$ (with possible repetitions) labeled  so that $a_1\le a_2\le \dots\le a_l$. It remains to check 
		\begin{equation}\label{e:ps<=>toG}
			a_{k+1}-a_k\in\mathscr R_{i_k,i_{k+1}} \quad\text{for all}\quad 1\le k<l.
		\end{equation}
		Let $\le$ denote the partial order on $\mathcal V$ induced from the arrows of $G$. Since $\mathscr R_{i,j}\subseteq\mathbb Z_{>0}$, we have
		\begin{equation*}
			\bs\omega_{i_k,a_k} \ngeq \bs\omega_{i_{k+1},a_{k+1}}\quad\text{for all}\quad 1\le k<l.
		\end{equation*} 
		Thus, since $G$ is totally ordered, the $a_k$ must be all distinct and  $\bs\omega_{i_1,a_1}<\bs\omega_{i_2,a_2}<\dots<\bs\omega_{i_l,a_l}$. 
		In particular, the only possible ordered path  from $\bs\omega_{i_k,a_k}$ to $\bs\omega_{i_{k+1},a_{k+1}}$ must have a single arrow having $\bs\omega_{i_k,a_k}$ and $\bs\omega_{i_{k+1},a_{k+1}}$ as head and tail, thus proving \eqref{e:ps<=>toG}.
	\end{proof}
	
	Let us also record the following lemma. 
	
	\begin{lem}
		Let $G$ be a pseudo $q$-factorization graph such that $V(\bs\pi_G)$ is a prime snake module. If $H$ is a connected subgraph of $G$, then $V(\bs\pi_H)$ is also a prime snake module.
	\end{lem}
	
	\begin{proof}
		It follows from  \eqref{e:subsnake} that $V(\bs\pi_H)$ is a snake module, while \Cref{l:equivdefpsnake} implies  it is prime if and only if $G_f(\bs\pi_H)$ is connected. Since $H$ is a pseudo $q$-factorization graph over $\bs\pi_H$ and, hence, it can be obtained from $G_f(\bs\pi_H)$ by a sequence of pure fusions such that in every step at least one of the vertices being fused is fundamental, if $G_f(\bs\pi_H)$ were disconnected, so would $H$ be, yielding a contradiction.
	\end{proof}

	\subsection{Proof of \Cref{t:orddecss}}\label{ss:orddecss}
	We start by remarking that, if $G_f(\bbs\pi)$ is not a singleton, then $G_f(\bs\pi)$ is connected. Indeed, let $v,w$ be vertices in $G$ and denote by $\bar{\mathcal F}$ the pseudo $q$-factorization map of $\bar G=G_f(\bbs\pi)$. Then, there are vertices $\bar v,\bar w$ in $\bar G$ such that
	\begin{equation*}
		\mathcal F(v)=\bar{\mathcal F}(\bar v) \quad\text{and}\quad \mathcal F(w)=\bar{\mathcal F}(\bar w). 
	\end{equation*}
	Since $\bar G$ is totally ordered, there is a directed path linking $\bar v$ to $\bar w$, say $\bar v_1=\bar v, \bar v_2, \dots, \bar v_l=\bar w$. By definition of $\bbs\pi$, there exist vertices $v_k$ in $G$ such that 
	\begin{equation}\label{e:lift}
		\mathcal F(v_k)=\bar{\mathcal F}(\bar v_k)\quad\text{for all}\quad 1\le k\le l,
	\end{equation}
	from where we conclude $v_1,\dots,v_l$ is a path in $G$. Since this conclusion is independent of the choice of $v_k$, we can choose $v_1=v$ and $v_l=w$. 
	
	Henceforth, if $H$ is a maximal totally ordered subgraph of $G$, we shall simply say $H$ is an mtos. The first key step is:

	\begin{lem}\label{l:maxtosgss}
		Suppose $\bs\pi$ has snake support and $H\triangleleft G=G_f(\bs\pi)$. Then, $H$ is an mtos if and only if $\bs\pi_H=\bbs\pi$. 
	\end{lem}
	
	\begin{proof}
		Let $\bar v_1,\dots,\bar v_l$ be the vertices of $\bar G=G_f(\bbs\pi)$ numbered so that $\bar v_1<\bar v_2<\cdots<\bar v_l$. If $l=1$, there is nothing to do. Thus, assume henceforth $l>1$, so $G$ is connected.
		
		Let $v_1,\dots,v_l$ be vertices in $G$ satisfying \eqref{e:lift} and let $H$ be the subgraph such that $\mathcal V_H = \{v_k: 1\le k\le l\}$. By definition, $\bs\pi=\bbs\pi$ and $H$ is totally ordered. If it were not maximal, there would exist $v\in\mathcal V_G$ such that the subgraph $H'$ whose vertex set is $\mathcal V_H\cup\{v\}$ is totally ordered. By definition of $\bbs\pi$, we must have $\mathcal F(v)=\mathcal F(v_k)$ for some $k$. But vertices having the same image under $\mathcal F$ are not linked by a directed path. Hence, they are not comparable in the partial order of $\mathcal V_G$, i.e., $H'$ is not totally ordered, yielding a contradiction. This proves that if $\bs\pi_H = \bbs\pi$, then $H$ is an mtos. 
		
		Conversely, if $H$ is an mtos, let $v_1,\dots, v_m$ be an enumeration of its vertices so that $v_1<v_2<\cdots<v_m$. The above argument shows we must have $\mathcal F(v_k)\ne\mathcal F(v_j)$ for all $1\le j<k\le m$. It then suffices to show $m=l$. Evidently, $m\le l$. If $m<l$, we split the argument in two cases:
		\begin{enumerate}[(i)]
			\item $\bar{\mathcal F}(\bar v_k)=\mathcal F(v_k)$ for all $1\le k\le m$;
			\item there exists $1\le k\le m$ such that  $\bar{\mathcal F}(\bar v_k)\ne\mathcal F(v_k)$.
		\end{enumerate}
		In case (i), choose $v_{m+1}$  such that $\mathcal F(v_{m+1}) = \bar{\mathcal F}(\bar v_{m+1})$ and let $H'$ be the subgraph whose vertices are $v_k, 1\le k\le m+1$. Then $H'$ is totally ordered and properly contains $H$, yielding a contradiction. In case (ii), let $k$ be the minimal value with the given property and chose $v\in\mathcal V_H$ such that $\mathcal F(v)=\bar{\mathcal F}(\bar v_k)$. If $k=1$, chose a directed path $w_1,\dots,w_r$ in $G$ with $w_1=v$ and $w_r=v_k$ and let $H'$ be the subgraph whose vertex-set is $\{w_s: 1\le s\le r\}\cup \mathcal V_H$. Then, $H'$ is totally ordered and contains $H$ properly, yielding a contradiction. Finally, if $k>1$, choose a directed path $w_1,\dots, w_r$ in $G$ such that $w_1=v_{k-1}, w_r=v_k$, and $w_1<w_2<\cdots<w_r$. Let $\bar w_k\in \mathcal V_{\bar G}$ be such that 
		\begin{equation*}
			\mathcal F(w_{s}) = \bar{\mathcal F}(\bar w_s) \quad\text{for all}\quad 1\le s\le r. 
		\end{equation*}
		Since $w_1,\dots, w_r$ is a directed path, then so is $\bar w_1,\dots, \bar w_r$ and, hence,
		\begin{equation*}
			\bar w_s = \bar v_{k+s-2} \quad\text{for all}\quad 1\le s\le r. 
		\end{equation*}
		If it were $r=2$, we would get a contradiction with (ii). But if $r>2$, the subgraph $H'$ whose vertex-set is $\mathcal V_H\cup\{w_s:1\le s\le r\}$ is a totally ordered graph properly containing $H$, yielding the usual contradiction. 
	\end{proof}
	
	It follows from \Cref{l:maxtosgss} that, if $H$ and $H'$ are mtos of $G$, then $H\cong H'$. Moreover, we also have
	\begin{equation}
		G\setminus H\cong G\setminus H'.
	\end{equation}

	\begin{lem}\label{l:cchss}
		If $H$ is an mtos of $G$, then every connected component of $G\setminus H$ has snake support.
	\end{lem}
	
	\begin{proof}
		Let $C$ be a connected component of $G\setminus H$. Then, $\bbs\pi_C$ divides $\bbs\pi_G$ and \eqref{e:subsnake} implies $V(\bbs\pi_C)$ is a snake module, which is prime by \Cref{l:equivdefpsnake}. 
	\end{proof}
	
	Fix an mtos of $G$, say $G_1$. A recursive application of Lemmas \ref{l:maxtosgss} and \ref{l:cchss}  produces a multicut $\mathcal G=G_1,\dots,G_l$ of $G$ which is an mtos-quochain. If $\mathcal G'=G_1',\dots,G'_m$ is another mtos-quochain, up to reordering, we can assume $G_2'$ arises from a connected component of $G\setminus G'_1$ which isomorphic to the connected component which gave rise to $G_2$ in $G\setminus G_1$. Hence, $G_2\cong G_2'$. An obvious inductive argument completes the proof that $G$ has a unique mtos-decomposition. 
	
	It remains to prove the ``moreover'' part of \Cref{t:orddecss}. To shorten notation, set
	\begin{equation*}
		\bs\pi_k = \bs\pi_{G_k} \quad\text{for}\quad 1\le k\le l.
	\end{equation*}
	Since $V(\bs\pi_k)$ is a prime snake module, it is prime. Let us proceed by induction on $l$ to prove that
	\begin{equation}\label{e:sspfexist}
		V(\bs\pi)\cong V(\bs\pi_1)\otimes \cdots \otimes V(\bs\pi_l),
	\end{equation} 
	which clearly starts if $l=1$. Thus, assume $l>1$ and let $\bs\pi' = \bs\pi\bs\pi_l^{-1}$. Evidently,
	\begin{equation*}
		\bbs\pi' = \bbs\pi=\bs\pi_1
	\end{equation*}
	and, hence, $G'=G_f(\bs\pi')$ has snake support. Moreover, $\mathcal G'=G_1,\dots,G_{l-1}$ is an mtos-quochain for $G'$ and it follows from the induction assumption that 
	\begin{equation}\label{e:sspfi}
		V(\bs\pi')\cong V(\bs\pi_1)\otimes \cdots \otimes V(\bs\pi_{l-1}).
	\end{equation} 
	Since snake modules are real, in light of \Cref{c:cyc}, it remains to show
	\begin{equation}\label{e:sspf}
		V(\bs\pi_k)\otimes V(\bs\pi_l) \quad\text{is simple for all}\quad 1\le k<l. 
	\end{equation}
	
	\begin{lem}\label{l:sspf}
		For each $1\le k<l$, one of the following holds:
		\begin{enumerate}[(i)]
			\item $\bs\pi_l$ divides $\bs\pi_k$;
			\item $G_k$ and $G_l$ are connected components of $G_k\otimes G_l$. 
		\end{enumerate}
	\end{lem}
	
	\begin{proof}
		If $l=2$, then (i) holds since $\bs\pi_s$ divides $\bs\pi_1$ for all $1\le s\le l$. 
		Assume (i) fails and let us prove (ii) holds. In particular, $k\ne 1$. By definition of mtos-quochain, $G_k$ is an mtos of $G_k\otimes\cdots\otimes G_l$. 
		
		Let $C_1,\dots, C_m$ be an enumeration of the connected components of $G_k\otimes\cdots\otimes G_l$. The construction of the mtos-quochains described above implies that, for each $1<s\le l$, there exists $1\le r\le m$ such that $G_s\triangleleft C_r$. Assume $G_k\triangleleft C_1$. It follows that, if $G_s\triangleleft C_1$ for some $k<s\le l$, then $\bs\pi_s$ divides $\bs\pi_k$. Hence, since (i) fails, $G_l$ cannot be a subgraph of $C_1$, which implies (ii) holds.
	\end{proof}
	
	If $k$ is such that \Cref{l:sspf}(ii) holds, then \eqref{e:sspf} is immediate. Otherwise, it follows from \Cref{p:testworksforsnakes}. 
	
	It remains to prove uniqueness part of the prime decomposition in \Cref{t:orddecss}. Thus, assume
	\begin{equation*}
		V(\bs\pi)\cong V(\bs\varpi_1)\otimes\cdots\otimes V(\bs\varpi_m) \quad\text{and $V(\bs\varpi_k)$ is prime for all } 1\le k\le m.  
	\end{equation*}
	In particular, $G'_k:=G_f(\bs\varpi_k)$ is connected and, hence, so is $G_f(\bbs\varpi_k)$. Since $\bbs\varpi_k$ divides $\bbs\pi$ for all $1\le k\le m$, it then follows that $G'_k$ has snake support. Using that $V(\bs\varpi_k)$ is prime once again together with \eqref{e:sspfexist} (with $\bs\varpi_k$ in place of $\bs\pi$), we conclude $V(\bs\varpi_k)$ is a prime snake module. Using the first part of \Cref{t:orddecss}, i.e., the uniqueness of mtos-decomposition for $G$, the uniqueness of the prime decomposition for $V(\bs\pi)$ follows if we show that, up to reordering if necessary, $G'_1,\dots,G'_m$ is an mtos-quochain. 
	This follows from  the following lemma.
	
	\begin{lem}
		Let $\mathcal G=G_1,\dots, G_l$ be a sequence of fundamental factorization graphs such $\bs\pi_{G_k}$ is a prime snake module for all $1\le k\le l$ and let $G=G_1\otimes\cdots\otimes G_k$. If no reordering of $\mathcal G$ is an mtos-quochain of $G$, there exists $1\le k<m\le l$ such that $V(\bs\pi_{G_k})\otimes V(\bs\pi_{G_m})$ is reducible.	
	\end{lem}
	
	\begin{proof}
		Up to reordering, assume $G_1$ is maximal in the set $\{G_s: 1\le s\le l\}$ which can be regarded as a subset of the set of subgraphs of $G$. Let us proceed by induction on $l\ge 2$. Suppose we have shown for $l=2$ and assume $l>2$. 
		
		Suppose first that $G_k$ is an mtos of $G$ for some $k$. By our assumption on $G_1$, we may assume $k=1$. Consider $\mathcal G'=G_2,\dots, G_l$. Then, $\mathcal G'$ satisfies all the assumptions and the inductive assumption completes the proof. Otherwise, there exists $m>1$, such that $G_m$ contains a vertex $v$ such that the subgraph with vertex-set $\mathcal V_{G_1}\cup\{v\}$ is totally ordered and properly contains $G_1$. For simplicity of notation, reorder so that $m=2$. It now suffices to show
		\begin{equation}\label{e:rtps}
			V(\bs\pi_{G_1})\otimes V(\bs\pi_{G_2}) \quad\text{is reducible.}
		\end{equation}
		Note this also proves that induction starts when $l=2$. We shall actually prove 
		\begin{equation}\label{e:rtps'}
			\wt_\ell V(\bs\pi_{G_1})\otimes V(\bs\pi_{G_2})\neq \wt_\ell V(\bs\pi_{G_1}\bs\pi_{G_2}),
		\end{equation}
		which clearly implies \eqref{e:rtps}.

		For proving \eqref{e:rtps'}, let $\bs\pi=\bs\pi_{G_1}\bs\pi_{G_2}$ and let $(\bs i,\bs a)\in I^r\times \mathbb Z^r$ be such that
		\begin{equation*}
			\bar{\bs\pi}=\bs\omega_{\bs i,\bs a} = \prod_{s=1}^r \bs\omega_{i_s,a_s} \quad\text{and}\quad a_1<\cdots <a_r. 
		\end{equation*}
		For $j\in\{1,2\}$ define $w_j\in \{0,1\}^{r}$ by requiring that $w_j(s)=1$ iff $\bs\omega_{i_s,a_s}\in \mathcal V_{G_j}$.  Our assumption on $G_1$ and $G_2$ implies that there exists $1\leq s_1< s_2\leq r$ such that 
		\begin{equation*}
			(w_1-w_2)(s_1)= \pm 1 \quad\text{and}\quad (w_1-w_2)(s_2)=\mp 1.
		\end{equation*}
		By choosing $s_1$ and $s_2$ such that $s_2-s_1$ is minimal, we have \begin{equation*}
			(w_1-w_2)(s)=0 \quad\text{for all}\quad s_1<s<s_2.
		\end{equation*}
		Given $j\in\{1,2\}$, set 
		\begin{equation*}
			\bs\pi_{j}^< = \prod_{\substack{s<s_1,\\ w_j(s)=1}} \bs\omega_{i_s,a_s}, \quad 
			\bs\pi_{j}^> = \prod_{\substack{s>s_2,\\ w_j(s)=1}} \bs\omega_{i_s,a_s}, \quad
			\bs\pi_{j} = \prod_{\substack{s_1\le s\le s_2,\\ w_j(s)=1}} \bs\omega_{i_s,a_s}.
		\end{equation*}
		In particular, $\bs\pi_{G_j} = \bs\pi_j^<\,\bs\pi_j\,\bs\pi_j^>$. The following argument extends the one used in \cite[Section 2.5]{bc:hokr} to the present context. 
		
		Let $l' = s_2-s_1+1>1$ and $(\bs i',\bs a') \in I^{l'}\times \mathbb Z^{l'}$ be such that 
		\begin{equation*}
			\bs\omega_{\bs i',\bs a'} = \prod_{s=s_1}^{s_2} \bs\omega_{i_s,a_s}. 
		\end{equation*}
		In particular, in the notation of \eqref{e:my:tsysred}, there exists $j,j'$ such that $\{j,j'\}=\{1,2\}$,
		\begin{equation*}
			\bs\pi_j = \bs\omega_{i',a',1,l'-1}, \quad \text{and} \quad \bs\pi_{j'} = \bs\omega_{i',a',2,l'},
		\end{equation*}
		and, hence, $V(\bs\pi_1)\otimes V(\bs\pi_2)$ is reducible. Moreover, it follows from \eqref{e:my:tsysnlw} that there exists $\bs\omega_j\in \wt_\ell V(\bs\pi_j)$ such that 
		\begin{equation}\label{e:mynotlw}
			\bs\omega_1\bs\omega_2\notin\wt_\ell V(\bs\pi_1\bs\pi_2).
		\end{equation} 
		Note \eqref{e:my:tsyslw}, applied to $\bs\pi_{G_j}$ in place of $\bs\omega_{\bs i,\bs a}$, implies
		\begin{equation*}
			\bs\pi_j^<\,\bs\omega_j({}^*\bs\pi_j^>)^{-1} \in V(\bs\pi_{G_j})
		\end{equation*}
		and, therefore,
		\begin{equation*}
			\bs\omega:=\bs\pi_1^{<}\,\bs\pi_2^<\,\bs\omega_1\,\bs\omega_2\,({}^*\bs\pi_1^>\,{}^*\bs\pi_2^>)^{-1}\in \wt_\ell V(\bs\pi_{G_1})\otimes V(\bs\pi_{G_2}).
		\end{equation*}
		Thus, in order to complete the proof, it suffices to check 
		\begin{equation*}
			\bs\omega\notin \wt_\ell V(\bs\pi_{G_1}\bs\pi_{G_2}).
		\end{equation*} 
		Indeed, by definition, we have $\bs\pi_1^{<}\bs\pi_2^< \lessdot \bs\pi_1\bs\pi_2 \lessdot \bs\pi_1^>\bs\pi_2^>$  and, therefore, an application of \eqref{e:rmhwtp} implies we have an epimorphism
		\begin{equation*}
			V(\bs\pi_1^>\bs\pi_2^>)\otimes V(\bs\pi_1\bs\pi_2)\otimes V(\bs\pi_1^{<}\bs\pi_2^<) \to V(\bs\pi_{G_1}\bs\pi_{G_2}). 
		\end{equation*}
		Hence, it suffices to show
		\begin{equation}\label{e:lwnitp}
			\bs\omega\notin \wt_\ell V(\bs\pi_1^>\bs\pi_2^>)\otimes V(\bs\pi_1\bs\pi_2)\otimes V(\bs\pi_1^{<}\bs\pi_2^<).
		\end{equation} 
		Note \eqref{e:lwbounds} implies $(^*\bs\pi_1^>\,{}^*\bs\pi_2^>)^{-1}$ and $\bs\pi_1^{<}\bs\pi_2^<$ occur in $\bs\omega$. Then, if \eqref{e:lwnitp} were false, an application of \Cref{l:extocc} would imply $\bs\omega_1\bs\omega_2\in\wt_\ell V(\bs\pi_1\bs\pi_2)$, yielding a contradiction with  \eqref{e:mynotlw}. 
	\end{proof}

	\subsection{The Case of Monochromatic Graphs}\label{ss:monoc}
	We now prove \Cref{p:equivinseg}, starting with:
	
	\begin{proof}[Proof of \Cref{l:moncto}]
		Assume $G=G_f(\bbs\pi)$ is connected and let $S=\{\bs\omega_k: 1\le k\le l\}$ be the set of fundamental factors of $\bs\pi$. Then, $\bs\omega_k = \bs\omega_{i,a_k}$ for some $a_k\in\mathbb Z$. We assume the enumeration is chosen so that $a_1<a_2<\cdots<a_l$. In particular
		\begin{equation*}
			r<s \quad\Rightarrow\quad (\bs\omega_{i,a_r},\bs\omega_{i,a_s})\notin \mathcal A_G.
		\end{equation*}
		Then, the connectedness of $G$ implies that, for each $1\le r< l$, there exists $r<s\le l$ such that $(\bs\omega_{i,a_s},\bs\omega_{i,a_r})\in \mathcal A_G$. It follows from \Cref{l:positiveppath}(a) that
		\begin{equation*}
			(\bs\omega_{i,a_k},\bs\omega_{i,a_m})\in \mathcal A_G \quad\text{for all}\quad r\le m\le k\le s,
		\end{equation*}
		which completes the proof.
	\end{proof}
	
	The comments before the statement of \Cref{p:equivinseg} shows we have already proved (i) $\Leftrightarrow$ (ii) $\Leftrightarrow$ (iii), while \Cref{c:topreserved} shows (iii) $\Rightarrow$ (iv) and the main result of \cite{ms:to} (Theorem 3.5.5) shows (iv) $\Rightarrow$ (v). On the other hand, (v) implies $G_f(\bs\pi)$ is connected and then, \Cref{l:moncto} implies it has snake support which, in light of \Cref{t:orddecss}, shows that (v) $\Rightarrow$ (ii). Thus, it suffices to show that, for monochromatic graphs, (iv) $\Rightarrow$ (iii) (this is easily seen to be false in general). 	
	
	Let $\bs\pi\in\mathcal P^+$ and assume $G(\bs\pi)$ is connected and $\sup(\bs\pi)=\{i\}$. In particular, there exists $a\in\mathbb Z, l\in\mathbb Z_{>0}$, $k_j\in\mathbb Z, m_j\in\mathbb Z_{>0}, 1\le j\le l$, such that
	\begin{equation*}
		\bs\pi = \prod_{j=1}^l (\bs\omega_{i,a+k_j})^{m_j} \quad\text{and}\quad k_j< k_{j'} \ \text{ if }\ j< j'.
	\end{equation*}
	Moreover, any other such expression differs by replacing $a$ by $a-k$ for some $k$, which replaces $k_j$ by $k_j+k$  and does not change $l$ or $m_j$. Thus, fix $a$ and let $\bs k(\bs\pi)$ be the sequence $(k_1,\dots,k_l)$. The proof of (iv) $\Rightarrow$ (iii) is clearly completed with the following lemma.
	
	\begin{lem}\label{l:to=>seg}
		If $G(\bs\pi)$ is totally ordered, then $m_j=1$ for all $1\le j\le l$ and $\bs k(\bs\pi)\in S_{i,n}$. \endd
	\end{lem}
	
	We need some preparation to prove this lemma. Given $\bs\omega,\bs\varpi\in\mathcal P^+$, we will write $\bs\omega|\bs\varpi$ to mean that $\bs\omega$ divides $\bs\varpi$. If $\bs\omega,\bs\varpi\in\mathcal{KR}$, and $\bs\omega|\bs\varpi$, there exist $i\in I, a,m\in\mathbb Z, r,s\in\mathbb Z_{>0},s\le r$, such that
	\begin{equation*}
		\bs\varpi = \bs\omega_{i,a,r}, \quad\bs\omega=\bs\omega_{i,a+m,s}
	\end{equation*}
	and,
	\begin{equation}\label{e:divcenter}
		m = r+s -2p \quad\text{for some}\quad s\le p\le r. 
	\end{equation}
	Given $\bs\pi\in\mathcal P^+$, recall that we denote by $\preccurlyeq$ the partial order on the vertex set of $G(\bs\pi)$ induced by the arrow structure.
	
	\begin{lem}\label{l:noto}
		Let  $\bs\pi\in\mathcal P^+$. If there exist  $q$-factors of $\bs\pi$, say $\bs\omega$ and $\bs\varpi$, such that $\bs\omega|\bs\varpi$, then $\bs\omega$ and $\bs\varpi$ are not related by  $\preccurlyeq$. 
	\end{lem}
	
	\begin{proof}
		Fix the notation leading to \eqref{e:divcenter}. If it were $\bs\omega\preccurlyeq\bs\varpi$, there would exist vertices 
		$v_1 = \bs\omega, v_2,\dots, v_k=\bs\varpi$ such that $(v_{j+1},v_j)$ are arrows for $1\le j<k$. In other words, we would have
		\begin{equation*}
			v_j = \bs\omega_{i,a_j,r_j} \quad\text{with}\quad a_{j+1}-a_j = m_j, \ \ m_j\in\mathscr R_{i,i}^{r_j,r_{j+1}}.
		\end{equation*}
		In particular, 
		\begin{equation*}
			m = -\sum_{j=1}^{k-1} m_j.
		\end{equation*}
		Writing $m_j = r_j+r_{j+1} - 2p_j$ and $d=d(i,\partial I)$, the condition $m_j\in\mathscr R_{i,i}^{r_j,r_{j+1}}$ is equivalent to
		\begin{equation*}
			-d\le p_j<\min\{r_j,r_{j+1}\}.
		\end{equation*} 
		Hence,
		\begin{equation*}
			m = s+r -2p \quad\text{with}\quad p = r+s-\sum_{j=1}^{k-1} p_j + \sum_{j=2}^{k-1} r_j.
		\end{equation*}
		Thus, 
		\begin{equation*}
			p = r+(s- p_1) + \sum_{j=2}^{k-1} (r_j - p_j) > p_1 > r
		\end{equation*}
		yielding a contradiction with \eqref{e:divcenter}. 
		
		Similarly, if it were $\bs\varpi\preccurlyeq\bs\omega$, by repeating the above argument with $v_1=\bs\varpi$ and $v_k=\bs\omega$, we would conclude 
		\begin{equation*}
			m = \sum_{j=1}^{k-1} m_j = s+r -2p \quad\text{with}\quad p = \sum_{j=1}^{k-1} p_j - \sum_{j=2}^{k-1} r_j.
		\end{equation*}
		Thus, 
		\begin{equation*}
			p = p_1 + \sum_{j=2}^{k-1} (p_j - r_j) < p_1 < s,
		\end{equation*}
		yielding a contradiction with \eqref{e:divcenter} again. 
	\end{proof}
	
	We will also need:
	
	\begin{lem}\label{l:ext11}
		Let $i\in I, r,s\in\mathbb Z_{>0}$ and assume $\bs\omega_{i,a,r}$ and $\bs\omega_{i,a+m,s}$ are the $q$-factors of $\bs\pi=\bs\omega_{i,a,r}\,\bs\omega_{i,a+m,s}$. If $m\in\mathscr R_{i,i}^{r,s}$, then $m-s+1-(r-1)\in\mathscr R_{i,i}$ or, equivalently, 
		$$(1-r,3-r,\dots, r-1, m-s+1,m-s+3,\dots,m+s-1)\in S_{i,n,r+s}.$$ 
	\end{lem}
	
	\begin{proof}
		Write $m=r+s-2p$ with $-d(i,\partial I)\le p<\min\{r,s\}$. Therefore,
		\begin{equation*}
			m-s+1-(r-1) = -2(p-1)\in \{2j: 1-\min\{r,s\}\le j\le d(i,\partial I)+1\}
		\end{equation*}
		and we need to show we must have $j\ge 1$. If it were $j<1$ or, equivalently, $p>0$, the union of the $q$-strings $(1-r,3-r,\dots, r-1)$ and $(m-s+1,m-s+3,\dots,m+s-1)$ would form a longer $q$-string, contradicting the assumption that $\bs\omega_{i,a,r}$ and $\bs\omega_{i,a+m,s}$ are the $q$-factors of $\bs\pi$.  
	\end{proof}
	
	Finally, we can give the:
	
	\begin{proof}[Proof of \Cref{l:to=>seg}]
		If it could be $m_j>1$ for some $j$, it would follow that $\bs\pi$ has two $q$-factors with a common root. The fact that they are $q$-factors implies the assumptions of \Cref{l:noto} are satisfied, yielding a contradiction with the assumption that $G(\bs\pi)$ is totally ordered.  Thus, it remains to show $\bs k(\bs\pi)\in S_{i,n}$, which follows from \Cref{l:ext11}.
	\end{proof}

\end{document}